\documentclass[reqno]{amsart}

\usepackage{tabu}
\usepackage{amssymb}
\usepackage{mathtools}
\usepackage{a4wide,amsmath}
\usepackage{mathrsfs}
\usepackage{amsthm}
\usepackage{dsfont}
\numberwithin{equation}{section}
\numberwithin{figure}{section}
\numberwithin{table}{section}
\usepackage{bbm}
\usepackage{subfig}
\usepackage{enumerate}
\usepackage{graphicx}		  
\usepackage{ifpdf}
\ifpdf
\DeclareGraphicsExtensions{.pdf,.eps,.jpg,.png}	
\usepackage[suffix=]{epstopdf}
\fi
\usepackage{xcolor}
\usepackage[utf8]{inputenc}
\usepackage{hyperref}
\hypersetup{hidelinks}

\long\def\MSC#1\EndMSC{\def\arg{#1}\ifx\arg\empty\relax\else
	{\narrower\noindent%
		{2020 Mathematics Subject Classification}: #1\\} \fi}
\long\def\PACS#1\EndPACS{\def\arg{#1}\ifx\arg\empty\relax\else
	{\narrower\noindent%
		{PACS numbers}: #1}\fi}
\long\def\KEY#1\EndKEY{\def\arg{#1}\ifx\arg\empty\relax\else
	{\narrower\noindent%
		Keywords: #1\\}\fi}

\theoremstyle{plain}
\newtheorem{theorem}{Theorem}[section]
\newtheorem{lemma}[theorem]{Lemma}
\newtheorem{proposition}[theorem]{Proposition}

\theoremstyle{definition}

\theoremstyle{remark}
\newtheorem{remark}[theorem]{Remark}

\newcommand{\norm}[1]{\lVert#1\rVert}
\newcommand{\abs}[1]{\lvert#1\rvert} 
\newcommand{\inner}[1]{\langle#1\rangle} 
\newcommand{\redel}{\mathop{\textup{Re}}}

\newcommand{\essinf}{\mathop{\textup{ess\,inf}}}

\newcommand{\spanm}{\mathop{\textup{span}}}

\newcommand{\I}{\mathrm{i}}    
\newcommand{\e}{\mathrm{e}}    
\newcommand{\di}{\mathrm{d}}   

\newcommand{\R}{\mathbb{R}}
\newcommand{\N}{\mathbb{N}}
\newcommand{\C}{\mathbb{C}}
\newcommand{\Z}{\mathbb{Z}}

\begin{document}
\title[Reconstruction of unbounded perturbations in 3D]{Linearised Calder\'on problem: Reconstruction of unbounded perturbations in 3D}
	
\author[H.~Garde]{Henrik Garde}
\address[H.~Garde]{Department of Mathematics, Aarhus University, Ny Munkegade 118, 8000 Aarhus C, Denmark.}
\email{garde@math.au.dk}
    
\author[M.~Hirvensalo]{Markus Hirvensalo}
\address[M.~Hirvensalo]{Department of Mathematics and Systems Analysis, Aalto University, P.O. Box~11100, 00076 Helsinki, Finland.}
\email{markus.hirvensalo@aalto.fi}
	
\begin{abstract}
    Recently an algorithm was given in [Garde \& Hyv\"onen, SIAM J.\ Math.\ Anal., 2024] for exact direct reconstruction of any $L^2$ perturbation from linearised data in the two-dimensional linearised Calder\'on problem. It was a simple forward substitution method based on a 2D~Zernike basis. We now consider the three-dimensional linearised Calder\'on problem in a ball, and use a 3D~Zernike basis to obtain a method for exact direct reconstruction of any $L^3$ perturbation from linearised data. The method is likewise a forward substitution, hence making it very efficient to numerically implement. Moreover, the 3D method only makes use of a relatively small subset of boundary measurements for exact reconstruction, compared to a full $L^2$ basis of current densities.
\end{abstract}	
\maketitle
	
\KEY
Calder\'on problem, 
linearisation,
reconstruction,
3D Zernike basis.
\EndKEY
	
\MSC
35R30, 35R25.
\EndMSC
 
\section{Introduction} \label{sec:intro}

Let $B$ be the unit ball in $\R^3$. For a conductivity coefficient $\gamma\in L^\infty(B;\R)$, with $\essinf \gamma > 0$, and a surface current density 
\begin{equation*}
   f \in L^2_\diamond(\partial B) = \{ g\in L^2(\partial B) \mid \inner{g,1}_{L^2(\partial B)} = 0 \},
\end{equation*} 
the continuum model for the conductivity problem states that the corresponding interior electric potential $u$ satisfies
\begin{equation} \label{eq:strong_form}
   -\nabla\cdot(\gamma\nabla u) = 0 \text{ in } B, \qquad \nu\cdot(\gamma\nabla u) = f \text{ on } \partial B,
\end{equation}
where $\nu$ is the exterior unit normal. The Lax--Milgram lemma yields a unique solution $u^\gamma_f$ to \eqref{eq:strong_form} in the space
\begin{equation*}
   H_\diamond^1(B) = \{w\in H^1(B) \mid w|_{\partial B} \in L^2_\diamond(\partial B)\}.
\end{equation*}
The Neumann-to-Dirichlet (ND) map $\Lambda(\gamma)f = u^\gamma_f|_{\partial B}$ is a compact self-adjoint operator in the space $\mathscr{L}(L^2_\diamond(\partial B))$, mapping any applied current density to the corresponding boundary potential measurement. 

The nonlinear forward map $\gamma\mapsto\Lambda(\gamma)$ is Fr\'echet differentiable with respect to complex-valued perturbations $\eta \in L^\infty(B)$. Let $F = D \Lambda(1; \eta)$ be the Fr\'echet derivative of $\Lambda$ at $\gamma \equiv 1$, with respect to perturbation $\eta$. If $u_f$ and $u_g$ are harmonic functions in $B$ with $f$ and $g$ as their Neumann traces, respectively, then $F\in \mathscr{L}(L^\infty(B),\mathscr{L}(L^2_\diamond(\partial B)))$ is characterised by
\begin{equation} \label{eq:frechet_redef}
   \inner{(F \eta) f, g}_{L^2(\partial B)} = -\int_B \eta \nabla u_f \cdot \overline{\nabla u_g} \,\di x,
\end{equation}
for $\eta \in L^\infty(B)$ and $f, g \in L^2_\diamond(\partial B)$. The linearised Calder\'on problem is:
\begin{equation*}
   \emph{Reconstruct $\eta$ from knowledge of $F\eta$}. 
\end{equation*}
This is in contrast to the (nonlinear) Calder\'on problem, to reconstruct a coefficient $\gamma$ from $\Lambda(\gamma)$.

By Proposition~\ref{prop:Ld}, $F$ continuously extends to an operator acting on perturbations in the larger space $L^3(B)\supset L^\infty(B)$, hence allowing for unbounded perturbations. Generally, for a bounded smooth domain $\Omega$ in $\R^d$, $F$ extends to perturbations in $L^d(\Omega)$. For the extension result, it is important that the Neumann conditions are in $L^2$ and not e.g.\ $H^{-1/2}$. This indicates that it may be possible to generalise the two-dimensional reconstruction method in \cite{Garde2024a} to three spatial dimensions, but for general perturbations in $L^3$ instead of $L^2$. The technique from \cite{Garde2024a} for obtaining stability for infinite-dimensional spaces of perturbations, however, is out of reach, as we cannot make use of the Hilbert--Schmidt topology in three spatial dimensions as outlined in \cite[Section~1.3]{Garde2024a} and \cite[Appendix~A]{Garde2022a}. 

Calder\'on's original injectivity proof for the linearised problem directly extends to $L^3(B)$, showing that $F\eta$ vanishes identically, if and only if, the Fourier transform for $\eta$ (zero-extended to $\R^3$) vanishes \cite{Calderon1980}. See also \cite{Ferreira2009,Ferreira2020,Sharafutdinov2009,Sjostrand2016} for additional uniqueness results in the linearised Calder\'on problem, including the case of partial data \cite{Ferreira2009,Sjostrand2016}.

In spherical coordinates, with radial distance $r$, polar angle $\theta$, and azimuthal angle $\varphi$, the 3D Zernike basis functions are
\begin{equation} \label{eq:basis}
   \psi_{\ell}^{k,m} (r, \theta, \varphi) = R_\ell^k(r) Y_\ell^m(\theta,\varphi), \quad k,\ell\in\N_0,\enskip m\in\Z_\ell,
\end{equation}
where 
\begin{equation*}
   \Z_\ell = \{-\ell,\dots,\ell\}.
\end{equation*}
Here $Y_\ell^m$ are the spherical harmonics of degree $\ell$ and order $m$, and $R_\ell^k$ are 3D radial Zernike polynomials (see Section~\ref{sec:basis} for definitions and notation).  $\{\psi_\ell^{k,m}\}_{\ell,k\in\N_0, m\in\Z_\ell}$ is an orthonormal basis for $L^2(B)$ by Proposition~\ref{prop:basis}, which we will use for expanding a perturbation $\eta\in L^3(B)\subset L^2(B)$.

Any $\eta\in L^3(B)$ can be reconstructed from the linearised data $F\eta$ via our main result below.
\begin{theorem} \label{thm:main}
    For any $\eta \in L^3(B)$, expanded as
    \begin{equation*}
		\eta = \sum_{k\in\N_0}\sum_{\ell\in\N_0}\sum_{m\in\Z_\ell} c_\ell^{k,m}\psi_\ell^{k,m},
    \end{equation*}
    for an $\ell^2$-sequence of coefficients $c_{\ell}^{k,m}$, then
    \begin{equation*}
		c_{\ell}^{k,m} = (Q_{\ell,0}^{k,m,k})^{-1}\Bigl( \inner{(F\eta)Y_{k+1}^0,Y_{\ell+k+1}^m}_{L^2(\partial B)} - \sum_{q=0}^{k-1}\sum_{s = 0}^{k-q} Q_{\ell,s}^{k,m,q}c_{\ell+2s}^{q,m}  \Bigr).
    \end{equation*}
    The scalars $Q_{\ell,s}^{k,m,q}$ are independent of $\eta$, and defined as 
    \begin{equation*}
        Q_{\ell,s}^{k,m,q} = (-1)^{m+1}\frac{\sqrt{2\ell+4q+4s+3}(k-s+1)(k-q-s+1)_q}{(k+1)(\ell+k+1)(\ell+k+s+\tfrac{5}{2})_q}G_{k+1,\ell+k+1,\ell+2s}^{0,-m,m},
    \end{equation*}
    using Pochhammer symbols (rising factorials) and the Gaunt coefficient
    \begin{equation*}
		G_{k+1,\ell+k+1,\ell+2s}^{0,-m,m} = \int_{\partial B} Y_{k+1}^0 Y_{\ell+k+1}^{-m}Y_{\ell+2s}^{m}\,\di S.
    \end{equation*}
\end{theorem}

The reason for leaving the Gaunt coefficient in $Q_{\ell,s}^{k,m,q}$, instead of inserting its exact value from \eqref{eq:gauntconstants}, is that Gaunt coefficients can be efficiently computed using finite element methods, or using direct implementations in libraries such as SymPy.

If we let 
\begin{equation*}
    \eta_k = \sum_{\ell\in\N_0}\sum_{m\in\Z_\ell} c_\ell^{k,m}\psi_\ell^{k,m},
\end{equation*}
then $\eta_k$ is an orthogonal projection of $\eta$ onto a particular infinite-dimensional subspace of $L^2(B)$, with $k=0$ giving the subspace of harmonic functions. Since
\begin{equation*}
    \eta = \sum_{k\in\N_0} \eta_k,
\end{equation*}
then Theorem~\ref{thm:main} implies that we can inductively reconstruct $\eta$ one orthogonal projection at a time. We have that $\eta_0$, i.e.\ the coefficients $\{c_{\ell}^{0,m}\}_{\ell\in\N_0,m\in\Z_\ell}$, is directly reconstructed from the data $(F\eta)Y_1^0$. Next $\eta_1$, i.e.\ the coefficients $\{c_{\ell}^{1,m}\}_{\ell\in\N_0,m\in\Z_\ell}$, is directly reconstructed from the data $(F\eta)Y_2^0$ and $\eta_0$. In general, we have that $\eta_k$ is directly reconstructed from $(F\eta)Y_{k+1}^0$ and $\eta_0,\dots,\eta_{k-1}$. Moreover, if one considers only finite measurements $\{(F\eta)Y_{k+1}^0\}_{k=0}^K$, then the formula in Theorem \ref{thm:main} still provides the exact reconstruction of the orthogonal projections $\eta_0,\dots,\eta_K$, due to the method essentially being a forward substitution.

Theorem~\ref{thm:main} is the 3D variant of the 2D reconstruction method in \cite[Theorem~1.3]{Garde2024a}, which made use of a 2D Zernike basis \cite{Zernike1934} to obtain a very similar triangular structure for the solution formulas of the coefficients. We note the interesting fact, that in 3D less data is required for the reconstruction, in the sense that the current densities $Y_{k+1}^{0}$ applied for the measurements do not vary in the $m$-index of the spherical harmonics. If all $Y_\ell^m$ spherical harmonics are used as current densities in the measurements, there would be an enormous redundancy since the range of $m$-indices grows as $2\ell+1$. This is unlike in 2D where all the Fourier basis functions are used. 

Another difference to the 2D method, is that the 3D method is likely less numerically stable since there is an extra sum, where previously computed coefficients are needed with a higher $\ell$-index than the $c_{\ell}^{k,m}$ that is being reconstructed. This comes from the fact, that the product of two exponentials is once again an exponential and the additional sum therefore does not appear in 2D by orthogonality in the Fourier basis. However in 3D the product of two spherical harmonics is a finite linear combination of spherical harmonics of various degrees. Nevertheless, we are still able to obtain good approximate reconstructions from inaccurate measurements.

\subsection{Article structure} 

We give a numerical example in Section~\ref{sec:numerical}, indicating how the inherent ill-posedness of the problem is observed in practice. Section~\ref{sec:basis} introduces the spherical harmonics, 3D radial Zernike polynomials, and related results. Section~\ref{sec:proofthm} proves Theorem~\ref{thm:main}. Finally, Appendix~\ref{sec:extension} extends the linearised problem from perturbations in $L^\infty$ to $L^d$ for spatial dimension $d$.

\section{A simple numerical example} \label{sec:numerical}

Consider the following translated and very localised Gaussian-type perturbation:
\begin{equation} \label{eq:etaexample}
    \eta(x) = \e^{-50\abs{x - (0,\frac{3}{10},0)^{\textup{T}}}^2} = \e^{-50r^2+30r\sin(\theta)\sin(\varphi)-\frac{9}{2}}.
\end{equation}

\begin{figure}[htb]
    \centering
    \includegraphics[width=.8\textwidth]{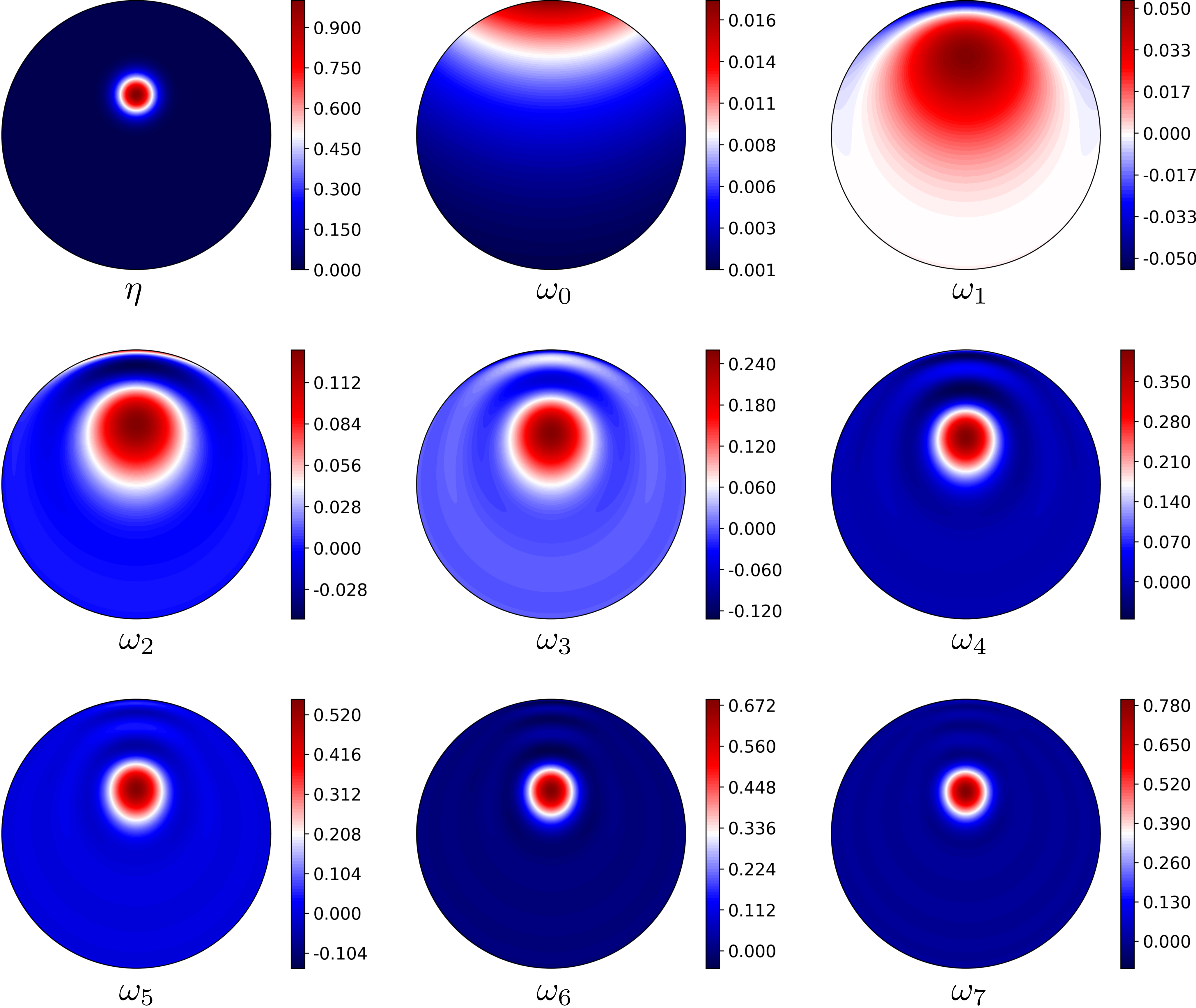}
    \caption{The perturbation $\eta$ from \eqref{eq:etaexample} and $\omega_K$ from \eqref{eq:omegaK} for $K = 0,\dots,7$. The plots are in the $xy$-plane.}
    \label{fig:etaplots}
\end{figure}

Figure~\ref{fig:etaplots} shows $\eta$ and the sum of the real part of the first few projections $\eta_k$,
\begin{equation} \label{eq:omegaK}
    \omega_K = \sum_{k=0}^K \redel(\eta_k),
\end{equation}
where $\eta_k$ is computed by evaluating inner products of $\eta$ with the 3D Zernike basis functions $\psi_{\ell}^{k,m}$, with $m\in\Z_\ell$ and up to a sufficiently high $\ell$-index (here up to $\ell = 30$). With accurate measurements, the full reconstruction will follow the pattern seen in Figure~\ref{fig:etaplots} for increasing $K$.

We may use \eqref{eq:Feta1} together with Mathematica to compute very accurate simulated measurements, evaluated correctly to 16 digits. A natural way of regularising the computations  is a $k$-dependent truncation in the $\ell$-indices, and still using all $m\in\Z_\ell$ (see also~\cite{Autio2024} for more details in a 2D setting). Thus let 
\begin{equation} \label{eq:omegatildeK}
    \widetilde{\omega}_K = \sum_{k=0}^{K}\sum_{\ell=0}^{\ell_k} \sum_{m\in\Z_\ell} \redel(\tilde{c}_{\ell}^{k,m}\psi_{\ell}^{k,m}),
\end{equation}
for $K = 0,\dots,7$, and with truncations $\ell_{0} = 20$, $\ell_1 = 18$, $\ell_2 = 16$, $\ell_3 = 14$, $\ell_4 = 12$, $\ell_5 = 10$, $\ell_6 = 8$, and $\ell_7 = 6$. The approximate coefficients $\tilde{c}_{\ell}^{k,m}$ are found using the formula in Theorem~\ref{thm:main}, but with the accurate measurement data simulated in Mathematica. The results can be seen in Figure~\ref{fig:reconplots1} and can be compared with Figure~\ref{fig:etaplots}; the approximate reconstructions are in fact nearly perfect with the highly accurate measurements, even for the truncated indices. Indeed the triangular nature of the method implies that the computed coefficients used for the approximate reconstructions in Figure~\ref{fig:reconplots1} are \emph{unaffected} by the truncation to a finite set of indices, which is very unique compared to other reconstruction methods.

\begin{figure}[htb]
    \centering
    \includegraphics[width=.8\textwidth]{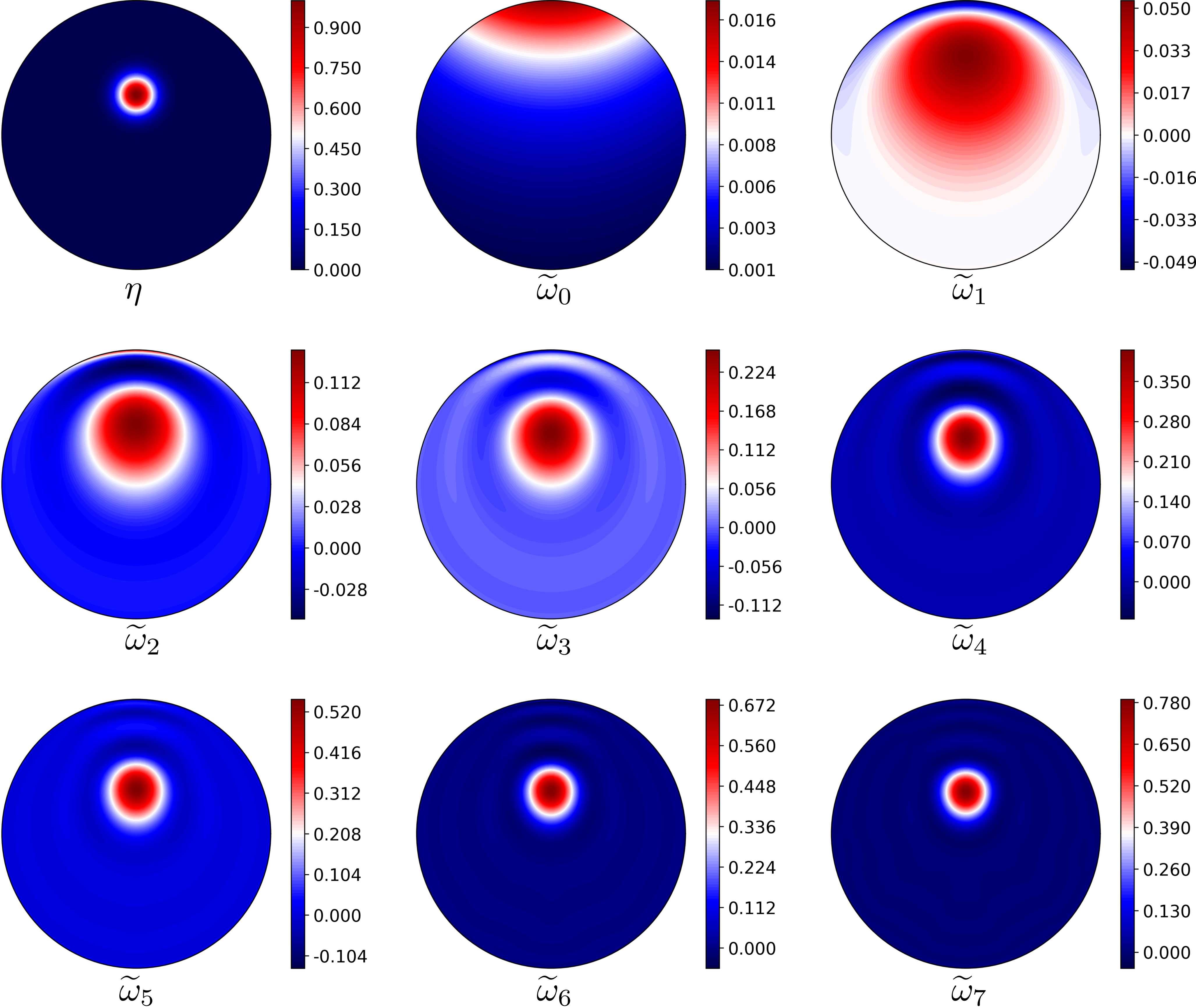}
    \caption{The perturbation $\eta$ from \eqref{eq:etaexample} and approximations $\widetilde{\omega}_K$ from \eqref{eq:omegatildeK} for $K = 0,\dots,7$ based on accurate measurements from Mathematica. The plots are in the $xy$-plane.}
    \label{fig:reconplots1}
\end{figure}

To indicate of how this looks in a more realistic setting, with some inaccuracies in the measurements, we simulate the measurements (including numerically computing the interior electric potentials) from a rather rough finite element (FE) discretisation. We use $\mathbb{P}_1$-elements on a tetrahedral mesh based on 6017 nodes (1026 boundary nodes) with the Python package scikit-fem~\cite{skfem}. The approximations are given by \eqref{eq:omegatildeK} for $K = 0,\dots,4$ and with truncations $\ell_{0} = 16$, $\ell_1 = 11$, $\ell_2 = 7$, $\ell_3 = 5$, and $\ell_4 = 3$. The approximate coefficients $\tilde{c}_{\ell}^{k,m}$ are found using the formula in Theorem~\ref{thm:main}, but this time with the approximate measurement data simulated by the rough FE model. The results can be seen in Figure~\ref{fig:reconplots2} and can be compared with the top two rows of Figure~\ref{fig:etaplots}. In practice with noisy measurements, this is realistically what can be achieved, and these coefficients appear to be numerically quite stable to compute.

\begin{figure}[htb]
    \centering
    \includegraphics[width=.8\textwidth]{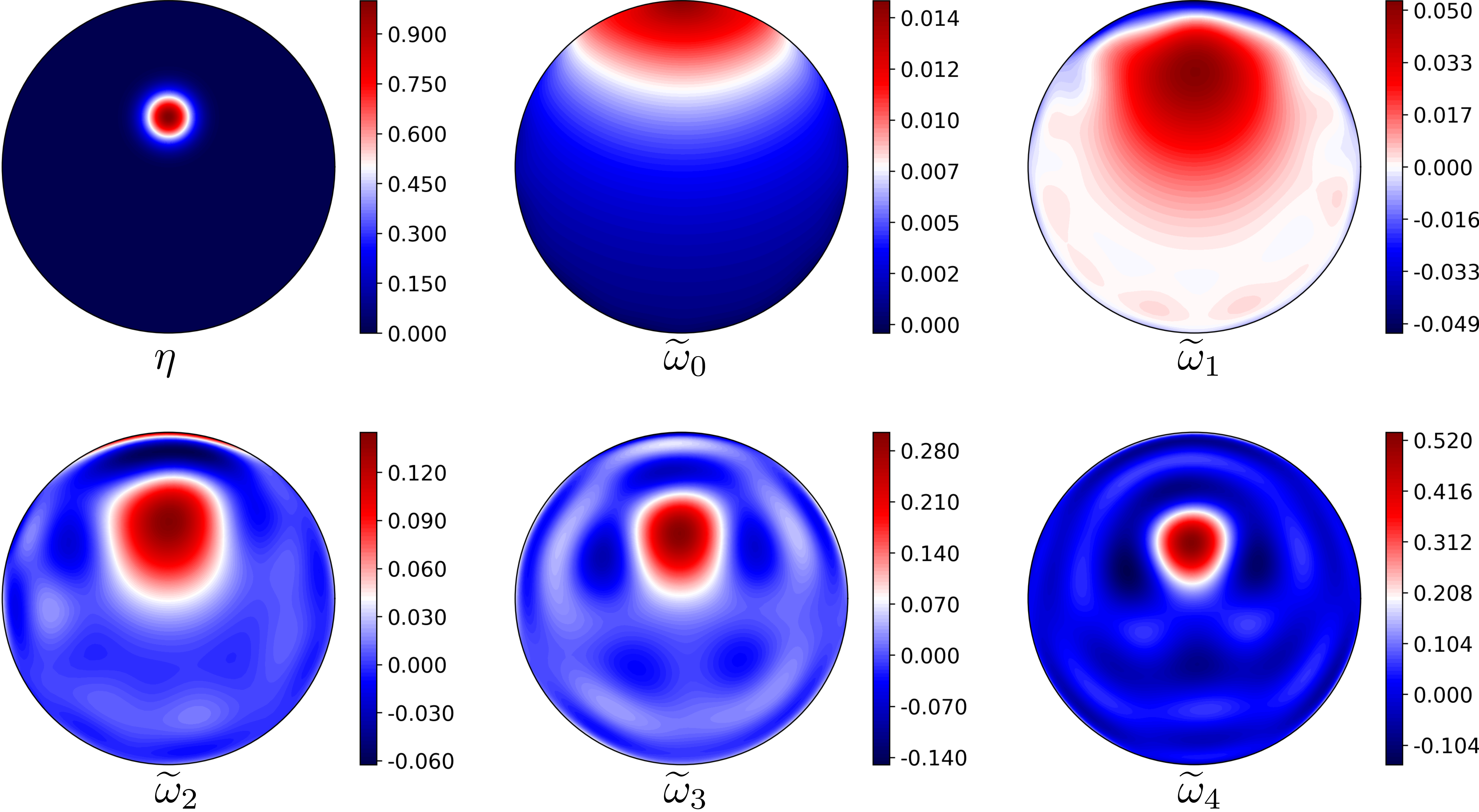}
    \caption{The perturbation $\eta$ from \eqref{eq:etaexample} and approximations $\widetilde{\omega}_K$ from \eqref{eq:omegatildeK} for $K = 0,\dots,4$ based measurements from a rough FE model. The plots are in the $xy$-plane.}
    \label{fig:reconplots2}
\end{figure}

For more on numerical computations and details on regularisation with this type of method, a numerical implementation of the 2D method from \cite{Garde2024a}, taking the triangular structure into account, can be found in \cite{Autio2024}. Even for truncated difference-data $\Lambda(1+\eta)-\Lambda(1)$, rather than linearised data, the (regularised) method consistently produces decent reconstructions. This can even be seen for data coming from practical electrode models, such as the complete electrode model, and for real-world measurements in an essentially 2D measurement setup. See also \cite{Allers91} for a numerical study using a 2D Zernike basis.
	
\section{Constructing the 3D Zernike orthonormal basis} \label{sec:basis}

This section defines spherical harmonics and 3D radial Zernike polynomials, and elaborates on related results needed for proving Theorem~\ref{thm:main}. 

\subsection{On spherical harmonics} 

We recall some facts about spherical harmonics; see \cite[Chapter~5]{Axler_2001} and \cite{Efthimiou_2014} for additional insights. See also \cite[§34]{NIST} for more info on Wigner $3j$ symbols.

A homogeneous polynomial $p:\R^3\to\C$ of degree $\ell$ (here in three spatial dimensions) satisfies $p(\lambda x) = \lambda^\ell p(x)$ for any $\lambda\in\R$, or equivalently
\begin{equation*}
    p(x) = \sum_{|\alpha| = \ell} c_\alpha x^\alpha,
\end{equation*}
for coefficients $c_\alpha\in \C$ with $\alpha\in \N_0^3$. Multi-index notation is used, where for $\alpha\in \N_0^3$ we have $x^\alpha = x_1^{\alpha_1} x_2^{\alpha_2} x_3^{\alpha_3}$ and $\abs{\alpha} = \alpha_1+\alpha_2+\alpha_3$. Let
\begin{equation*} 
    \mathcal{A}_\ell = \Bigl\{ p : \R^3 \to \C \mid p(x) = \sum_{|\alpha| = \ell} c_\alpha x^\alpha, \; c_\alpha \in \C, \; \Delta p = 0 \Bigr\}
\end{equation*}
be the space of \emph{harmonic} homogeneous polynomials of degree $\ell$. 

The space of spherical harmonics of degree $\ell$ is defined as the restriction of the harmonic homogeneous polynomials of degree $\ell$ to the unit sphere:
\begin{equation*} 
    \mathcal{H}_\ell = \{ p|_{\partial B} \mid p \in \mathcal{A}_\ell \}.
\end{equation*}
We have $\dim(\mathcal{H}_\ell) = \dim(\mathcal{A}_\ell) = 2\ell+1$ in three spatial dimensions.

A standard orthonormal basis for $\mathcal{H}_\ell$ comprises the spherical harmonics of degree~$\ell$ and order~$m$, $\mathcal{B}_\ell = \{Y_\ell^m\}_{m\in\Z_\ell}$, given by 
\begin{equation*}
    Y_\ell^m(\theta,\varphi) = \sqrt{\frac{2\ell+1}{4\pi}\frac{(\ell-m)!}{(\ell+m)!}}P^m_\ell(\cos(\theta))\e^{\I m\varphi}, \quad \ell\in\N_0, \enskip m\in\Z_\ell,
\end{equation*}
where $P_\ell^m$ is an \emph{associated Legendre polynomial} (with the Condon--Shortley phase). A symmetry condition is satisfied:
\begin{equation} \label{eq:conjsphharm}
    \overline{Y_\ell^m} = (-1)^m Y_\ell^{-m}.
\end{equation}

The $\mathcal{H}_\ell$-spaces are the eigenspaces for the Laplace–Beltrami operator $\Delta_{\partial B}$, with
\begin{equation} \label{eq:LBeig}
    \Delta_{\partial B}g = -\ell(\ell+1)g, \quad g\in\mathcal{H}_\ell.
\end{equation}
Thus there is the following orthogonal decomposition:
\begin{equation*}
    L^2(\partial B) = \bigoplus_{\ell=0}^\infty \mathcal{H}_\ell,
\end{equation*}
from which it holds that $\bigcup_{\ell=0}^\infty \mathcal{B}_\ell$ is an orthonormal basis for $L^2(\partial B)$, and $\bigcup_{\ell=1}^\infty \mathcal{B}_\ell$ is an orthonormal basis for $L^2_\diamond(\partial B)$.

Products of spherical harmonics have finite expansions in terms of spherical harmonics, and the Wigner $3j$ symbols, together with \eqref{eq:conjsphharm}, can be used for computing the inner products (also called Gaunt coefficients) \cite[§34.3(vii) Eq.~34.3.22]{NIST}:
\begin{equation} \label{eq:gaunt}
    \int_{\partial B} Y_{\ell_1}^{m_1}Y_{\ell_2}^{m_2}Y_{\ell_3}^{m_3}\,\di S = G_{\ell_1,\ell_2,\ell_3}^{m_1,m_2,m_3}
\end{equation}
with
\begin{equation} \label{eq:gauntconstants}
    G_{\ell_1,\ell_2,\ell_3}^{m_1,m_2,m_3} = \sqrt{\frac{(2\ell_1+1)(2\ell_2+1)(2\ell_3+1)}{4\pi}}
    \begin{pmatrix}
        \ell_1 & \ell_2 & \ell_3 \\
        0 & 0 & 0
    \end{pmatrix}
    \begin{pmatrix}
        \ell_1 & \ell_2 & \ell_3 \\
        m_1 & m_2 & m_3
    \end{pmatrix}.
\end{equation}
The integral in \eqref{eq:gaunt} vanishes if the following conditions are not met:
\begin{equation} \label{eq:gauntindices}
    \abs{\ell_1-\ell_2}\leq \ell_3 \leq \ell_1+\ell_2, \quad \sum_{j=1}^3 m_j = 0, \quad \text{and} \quad \sum_{j=1}^3\ell_j \text{ is even}.
\end{equation}
\begin{enumerate}[(i)]
    \item The first condition in \eqref{eq:gauntindices} is known as the \emph{triangle condition} for $3j$ symbols, and any $3j$ symbol vanishes if this condition is not met (see \cite[§34.2]{NIST} and \cite[Section~2.1]{Pain2020}).
    \item If the second condition in \eqref{eq:gauntindices} is not met, then the second $3j$ symbol in \eqref{eq:gauntconstants} vanishes \cite[§34.2]{NIST}, although there are also other non-trivial zeros of this $3j$ symbol \cite{Pain2020}.
    \item When the triangle condition holds, the third condition in \eqref{eq:gauntindices} is not met, if and only if, the first $3j$ symbol in \eqref{eq:gauntconstants} vanishes (see \cite[§34.3(i) Eq.~34.3.5]{NIST} and \cite[Eqs.~(47) and~(48)]{Pain2020}).
\end{enumerate}

One can combine \eqref{eq:conjsphharm} and \eqref{eq:gaunt} with the following lemma, to give finite expansions in terms of spherical harmonics for the product of gradients of spherical harmonics. In the following $\nabla_{\partial B}$ denotes a surface gradient on $\partial B$. As we have been unable to find the result published, we give a proof for the sake completion\footnote{The proof is inspired by ideas from a physics blog: \url{https://hyad.es/vsphint}}.

\begin{lemma} \label{lemma:3gradient}
    Let $\ell_j\in \N_0$ and $g_{\ell_j} \in \mathcal{H}_{\ell_j}$ for $j\in\{1,2,3\}$. Then
    \begin{equation*}
        \int_{\partial B} (\nabla_{\partial B} g_{\ell_1} \cdot \nabla_{\partial B} g_{\ell_2}) g_{\ell_3} \, \di S = \frac{\ell_1 (\ell_1 + 1) + \ell_2 (\ell_2 + 1) - \ell_3 (\ell_3 + 1)}{2} \int_{\partial B} g_{\ell_1} g_{\ell_2} g_{\ell_3} \, \di S.
    \end{equation*}
\end{lemma}
\begin{proof}
    Below $\nabla$ denotes a gradient in $B$. We extend $g_{\ell_1}$, $g_{\ell_2}$, and $g_{\ell_3}$ to $B\setminus\{0\}$, constant in the radial direction. Using spherical coordinates gives
    \begin{equation} \label{eq:gip}
        \nabla g_{\ell_i}\cdot\nabla g_{\ell_j} = r^{-2}\nabla_{\partial B} g_{\ell_i}\cdot \nabla_{\partial B} g_{\ell_j},
    \end{equation}
    and, using \eqref{eq:LBeig},
    \begin{equation} \label{eq:deltag}
        \Delta g_{\ell_j} = r^{-2} \Delta_{\partial B} g_{\ell_j} = -r^{-2}\ell_j (\ell_j + 1) g_{\ell_j}.
    \end{equation}
    In the computations below, one can replace $B$ by $B\setminus B_\epsilon$ for an origin-centered ball $B_\epsilon$ with radius $\epsilon>0$, and let $\epsilon\to 0$. Since the considered functions are integrable on $B$, the limiting process gives precisely the results below, so we avoid these unnecessary technicalities in what follows. 
    	
    Define
    \begin{equation} \label{eq:Edef}
        E_{\ell_1,\ell_2,\ell_3} = \int_{B} (\nabla g_{\ell_1} \cdot \nabla g_{\ell_2}) g_{\ell_3} \, \di x.
    \end{equation}
    Consider the first integral on the right hand-side below: Integrating by parts, noting that the boundary term vanishes due to a vanishing normal derivative of $g_{\ell_2}$, and using the product rule, gives
    \begin{equation} \label{eq:Ecyclic}
        E_{\ell_1,\ell_2,\ell_3} = -\int_B g_{\ell_1}(\Delta g_{\ell_2})g_{\ell_3}\,\di x - E_{\ell_2,\ell_3,\ell_1}.
    \end{equation}
    Note that the assumptions needed for arriving at \eqref{eq:Ecyclic} are symmetric in terms of the indices $\ell_1,\ell_2,\ell_3$, and they can therefore be permuted to arrive at other such formulas. Thus we may use the cyclic property \eqref{eq:Ecyclic} in the following way:
    \begin{align}
        E_{\ell_1,\ell_2,\ell_3} &= -\int_B g_{\ell_1}(\Delta g_{\ell_2})g_{\ell_3}\,\di x - E_{\ell_2,\ell_3,\ell_1} \notag \\
   		&= -\int_B g_{\ell_1}(\Delta g_{\ell_2})g_{\ell_3}\,\di x + \int_B g_{\ell_2}(\Delta g_{\ell_3})g_{\ell_1}\,\di x + E_{\ell_3,\ell_1,\ell_2} \notag \\
   		&= -\int_B g_{\ell_1}(\Delta g_{\ell_2})g_{\ell_3}\,\di x + \int_B g_{\ell_2}(\Delta g_{\ell_3})g_{\ell_1}\,\di x - \int_B g_{\ell_3}(\Delta g_{\ell_1})g_{\ell_2}\,\di x - E_{\ell_1,\ell_2,\ell_3}. \label{eq:Eiter}
    \end{align}
    Combining \eqref{eq:deltag} and \eqref{eq:Eiter} gives
    \begin{align*}
        E_{\ell_1,\ell_2,\ell_3} &= \frac{\ell_1 (\ell_1 + 1) + \ell_2 (\ell_2 + 1) - \ell_3 (\ell_3 + 1)}{2} \int_{B} r^{-2}g_{\ell_1} g_{\ell_2} g_{\ell_3} \, \di x \\
   		&= \frac{\ell_1 (\ell_1 + 1) + \ell_2 (\ell_2 + 1) - \ell_3 (\ell_3 + 1)}{2} \int_{\partial B} g_{\ell_1} g_{\ell_2} g_{\ell_3} \, \di S.
    \end{align*}    	
    The proof is concluded by using \eqref{eq:gip} and \eqref{eq:Edef}:
    \begin{equation*}
        E_{\ell_1,\ell_2,\ell_3} = \int_{\partial B} (\nabla_{\partial B} g_{\ell_1} \cdot \nabla_{\partial B} g_{\ell_2}) g_{\ell_3} \, \di S. \qedhere
    \end{equation*}
\end{proof}

\subsection{On 3D radial Zernike polynomials}
    
Let $R_\ell^k$ be the 3D radial Zernike polynomial
\begin{equation} \label{eq:radial_pols}
    R_\ell^k(r) = \sqrt{2 \ell + 4 k + 3}\sum^k_{s = 0} (-1)^s\binom{k}{s}\binom{\ell+2k-s+\tfrac{1}{2}}{k} r^{\ell+2k-2s}, \quad \ell, k \in\N_0,
\end{equation}
where the generalised binomial coefficient is
\begin{equation*}
    \binom{x}{n} = \frac{\Gamma(x+1)}{n!\,\Gamma(x-n+1)}, \quad x > n-1,\enskip n\in \N_0. 
\end{equation*}
In \cite{Mathar2009} the 3D radial Zernike polynomials above are annotated as $R^{(\ell)}_{\ell+2k}$, similar to the 2D radial Zernike polynomials in \cite{Garde2024a}. We have modified the notation to better fit with the spherical harmonics, however one should keep this notational difference in mind. 

For the next part, consider the Pochhammer symbol (rising factorial):
\begin{equation*}
    (x)_n = \frac{\Gamma(x+n)}{\Gamma(x)}, \quad x > 0,\enskip n\in \N_0.
\end{equation*}
The coefficients in \eqref{eq:radial_pols} are the same as those in \cite{Mathar2009}: By writing out the four binomial coefficients in \cite{Mathar2009}, cancelling recurring terms, and collecting some terms in Pochhammer symbols (using that $\Gamma(n+1) = n!$ for $n\in\N_0$), gives
\begin{align*}
    &\frac{(-1)^s\sqrt{2 \ell + 4 k + 3}}{4^k}\binom{\ell+2k}{\ell}^{-1} \binom{\ell+2k}{s}\binom{\ell+k-s}{\ell}\binom{2\ell+4k-2s+1}{2k} \\
    &= \frac{(-1)^s\sqrt{2 \ell + 4 k + 3}}{k!\,4^k}\binom{k}{s}\frac{(2\ell+2k-2s+2)_{2k}}{(\ell+k-s+1)_{k}}.
\end{align*}
Using the duplication formula for Pochhammer symbols \cite[§5.2(iii) Eq.~5.2.8]{NIST} (a consequence of the Legendre duplication formula for $\Gamma$-functions)
\begin{equation*}
    (2x)_{2k} = 4^k(x)_k(x+\tfrac{1}{2})_k,
\end{equation*}
combined with 
\begin{equation*}
    \frac{(x)_k}{k!} = \binom{x+k-1}{k},
\end{equation*}
gives the coefficients in \eqref{eq:radial_pols}.

The 3D radial Zernike polynomials with the same index $\ell$ are orthonormal in the weighted space $L^2_{r^2}((0,1))$ \cite[Eq.~(33)]{Mathar2009}, 
\begin{equation*}
    \inner{R_\ell^k,R_\ell^{k'}}_{L^2_{r^2}((0,1))} = \int_0^1 R_\ell^k(r)R_\ell^{k'}\!(r)r^2\,\di r = \delta_{k,k'}.
\end{equation*}
For $p\in\N_0$ there is the following finite expansion (cf.~\cite[Eqs.~(39) and~(40)]{Mathar2009} with the notational differences, mentioned above, in mind):
\begin{equation} \label{eq:monomialexpansion}
    r^{\ell+2p} = \sum_{q = 0}^{p}\chi_{\ell}^{p,q}R_{\ell}^q(r),
\end{equation}
with
\begin{equation} \label{eq:chi}
    \chi_{\ell}^{p,q} = \inner{r^{\ell+2p},R_{\ell}^q}_{L^2_{r^2}((0,1))} = \frac{\sqrt{2\ell+4q+3}(p-q+1)_q}{(2\ell+2p+3)(\ell+p+\tfrac{5}{2})_q}
\end{equation}
in terms of Pochhammer symbols.
\begin{proposition} \label{prop:basis}
    $\{\psi_{\ell}^{k,m}\}_{\ell,k\in\N_0, m\in\Z_\ell}$ from \eqref{eq:basis} is an orthonormal basis for $L^2(B)$.
\end{proposition}
\begin{proof}
    We have already argued that $\{\psi_{\ell}^{k,m}\}_{\ell,k\in\N_0, m\in\Z_\ell}$ is an orthonormal set in $L^2(B)$, so what remains is to prove density. Since $\{Y_{\ell}^m\}_{\ell\in\N_0,m\in\Z_\ell}$ is an orthonormal basis for $L^2(\partial B)$, it suffices to prove that $r^n$ belongs to
    \begin{equation*}
        \overline{\spanm\{R_{\ell}^{k} \mid k\in\N_0\}}
    \end{equation*}
    for any $n,\ell\in\N_0$. By \eqref{eq:monomialexpansion} we have that $r^{\ell+2p} \in \spanm\{R_{\ell}^{k} \mid k\in\N_0\}$ for $p\in \N_0$. Hence, the proof reduces to approximating $r^n$ in $L^2((0,1))$ by polynomials in $\spanm\{r^{\ell+2p} \mid p\in\N_0\}$. This is possible by \cite[Lemma~6.1]{Garde2024a}.
\end{proof}
Note in particular, that for $k=0$ then $\{\psi_{\ell}^{0,m}\}_{\ell\in\N_0, m\in\Z_\ell}$ are the regular solid harmonics, the standard orthonormal basis for $L^2$ harmonic functions in $B$.

\section{Proof of Theorem~\ref{thm:main}} \label{sec:proofthm}
	
For $\hat{\ell}\in\N$ and $\hat{m}\in\Z_{\hat{\ell}}$,
\begin{equation*}
    u_{\hat{\ell}}^{\hat{m}}(r,\theta,\varphi) = \frac{1}{\hat{\ell}}r^{\hat{\ell}} Y_{\hat{\ell}}^{\hat{m}}(\theta,\varphi)
\end{equation*}
is harmonic and has Neumann trace $Y_{\hat{\ell}}^{\hat{m}}\in L^2_\diamond(\partial B)$. Hence for $\ell_j\in\N$ and $m_j\in\Z_{\ell_j}$, and by \eqref{eq:conjsphharm}, we have in terms of the spherical coordinates:
\begin{equation} \label{eq:prodsurfacegrad}
    \nabla u_{\ell_1}^{m_1} \cdot \overline{\nabla u_{\ell_2}^{m_2}} = (-1)^{m_2}r^{\ell_1+\ell_2-2}(Y_{\ell_1}^{m_1}Y_{\ell_2}^{-m_2} + \tfrac{1}{\ell_1\ell_2}\nabla_{\partial B}Y_{\ell_1}^{m_1}\cdot\nabla_{\partial B}Y_{\ell_2}^{-m_2}).
\end{equation}
By Lemma~\ref{lemma:3gradient}, \eqref{eq:gaunt}, and the conditions in \eqref{eq:gauntindices}, the parenthesis at the end of \eqref{eq:prodsurfacegrad} can be expanded by a \emph{finite} number of spherical harmonics. In particular \eqref{eq:prodsurfacegrad} is in $L^2(B)$.

Now fix $\ell,k\in\N_0$ and $m\in\Z_\ell$. Since $\eta\in L^3(B)\subset L^2(B)$, there is an $\ell^2$-sequence of coefficients $c_{\ell'}^{k',m'}$, with $\ell',k'\in\N_0$ and $m'\in\Z_{\ell'}$, such that
\begin{equation*}
    \eta = \sum_{\ell',k' \in \N_0}\sum_{m'\in\Z_{\ell'}}c_{\ell'}^{k',m'}\psi_{\ell'}^{k',m'}.
\end{equation*}
The goal is to establish a direct formula for the coefficient $c_{\ell}^{k,m}$ in terms of $F\eta$ and coefficients with smaller $k$-indices. 

We consider a particular choice of Neumann boundary conditions, and use \eqref{eq:frechet_redef} and \eqref{eq:prodsurfacegrad} to get
\begin{equation} \label{eq:Feta1}
    \inner{(F\eta)Y_{k+1}^0,Y_{\ell+k+1}^m}_{L^2(\partial B)} = -\int_B \eta \nabla u_{k+1}^{0} \cdot \overline{\nabla u_{\ell+k+1}^{m}}\,\di x = (-1)^{m+1}\int_B \eta r^{\ell+2k}\Phi_{\ell}^{k,m}\,\di x,
\end{equation}
with
\begin{equation*}
    \Phi_{\ell}^{k,m} = Y_{k+1}^{0}Y_{\ell+k+1}^{-m} + \tfrac{1}{(k+1)(\ell+k+1)}\nabla_{\partial B}Y_{k+1}^{0}\cdot\nabla_{\partial B}Y_{\ell+k+1}^{-m}.
\end{equation*} 
Inserting the series for $\eta$ into \eqref{eq:Feta1} gives
\begin{align} \label{eq:Feta2}
    \inner{(F\eta)Y_{k+1}^0,Y_{\ell+k+1}^m}_{L^2(\partial B)} \hspace{-1.5cm}& \notag\\
    &= (-1)^{m+1}\sum_{\ell',k' \in \N_0}\sum_{m'\in\Z_{\ell'}}c_{\ell'}^{k',m'}\inner{r^{\ell+2k},R_{\ell'}^{k'}}_{L^2_{r^2}((0,1))}\int_{\partial B}\Phi_{\ell}^{k,m}Y_{\ell'}^{m'}\di S.
\end{align}

\begin{remark}
    A subtle, but important, point is that $F$ is continuous with respect to $L^3(B)$, so the integral in \eqref{eq:frechet_redef} is a dual pairing between $L^3(B)$ for $\eta$ and $L^{3/2}(B)$ for the product of gradients. But the series for $\eta$ converges in $L^2(B)$. Choosing the Neumann conditions to be spherical harmonics is therefore essential, such that \eqref{eq:prodsurfacegrad} is in fact an $L^2(B)$-function, as we argued above. This implies that the integral instead acts as the $L^2(B)$ inner product which, by continuity, enables that the summation can be moved outside the integral in \eqref{eq:Feta2}. 
    
    It may well be that if $\eta\in L^3(B)$ then the series also converges in $L^3(B)$; such a result e.g.\ holds for Fourier series, with convergence in $L^p$ for functions in $L^p$ with $p\in(1,\infty)$. However, we prefer to avoid what is likely a technical endeavour of proving such a result.
\end{remark}

Lemma~\ref{lemma:3gradient} gives
\begin{equation} \label{eq:Phiint}
    \int_{\partial B}\Phi_{\ell}^{k,m}Y_{\ell'}^{m'}\,\di S = \tau_{\ell,\ell'}^k\int_{\partial B} Y_{k+1}^{0}Y_{\ell+k+1}^{-m}Y_{\ell'}^{m'}\,\di S,
\end{equation}
with
\begin{align}
    \tau_{\ell,\ell'}^k &= 1+\frac{(k+1)(k+2)+(\ell+k+1)(\ell+k+2)-\ell'(\ell'+1)}{2(k+1)(\ell+k+1)} \notag \\
    &=\frac{(\ell+2k+2-\ell')(\ell+2k+3+\ell')}{2(k+1)(\ell+k+1)}. \label{eq:tau}
\end{align}
From \eqref{eq:gauntindices} then \eqref{eq:Phiint} vanishes unless $m' = m$. It also vanishes unless $\ell+2k+2+\ell'$ is even, i.e., $\ell+\ell'$ must be even. Finally we need $\ell\leq \ell' \leq \ell+2k+2$, and actually $\ell\leq \ell' \leq \ell+2k$ because $\tau_{\ell,\ell'}^k = 0$ for $\ell' = \ell+2k+2$ and $\ell+\ell'$ is odd for $\ell' = \ell+2k+1$. This reduces to the cases
\begin{equation*}
    m' = m \quad \text{and} \quad \ell' = \ell+2s, \quad s\in\{0,\dots,k\}.
\end{equation*}
Hence, using these indices and \eqref{eq:gaunt} simplifies \eqref{eq:Feta2} to
\begin{equation} \label{eq:Feta3}
    \inner{(F\eta)Y_{k+1}^0,Y_{\ell+k+1}^m}_{L^2(\partial B)} = \sum_{k'\in\N_0}\sum_{s=0}^k c_{\ell+2s}^{k',m}\inner{r^{\ell+2k},R_{\ell+2s}^{k'}}_{L^2_{r^2}((0,1))} D_{\ell,s}^{k,m},
\end{equation}
with
\begin{equation} \label{eq:D}
    D_{\ell,s}^{k,m} = (-1)^{m+1}\tau_{\ell,\ell+2s}^k G_{k+1,\ell+k+1,\ell+2s}^{0,-m,m}.
\end{equation}
Writing $r^{\ell+2k} = r^{\ell+2s + 2(k-s)}$ for $s\in\{0,\dots,k\}$, implies that we can use \eqref{eq:monomialexpansion} to write
\begin{equation*}
    r^{\ell+2k} = \sum_{q = 0}^{k-s}\chi_{\ell+2s}^{k-s,q} R_{\ell+2s}^{q}(r).
\end{equation*}
By orthonormality of $R_{\ell+2s}^{k'}$ and $R_{\ell+2s}^q$, their inner products vanish unless $k' = q$, which from \eqref{eq:Feta3} gives
\begin{equation*} 
    \inner{(F\eta)Y_{k+1}^0,Y_{\ell+k+1}^m}_{L^2(\partial B)} = \sum_{s=0}^k\sum_{q = 0}^{k-s}\chi_{\ell+2s}^{k-s,q} D_{\ell,s}^{k,m}c_{\ell+2s}^{q,m} = \sum_{q=0}^k\sum_{s = 0}^{k-q}\chi_{\ell+2s}^{k-s,q} D_{\ell,s}^{k,m}c_{\ell+2s}^{q,m},
\end{equation*} 
where the summation indices could be swapped, as it corresponds to summing over the same triangular pairs of indices. Writing
\begin{equation} \label{eq:Q}
    Q_{\ell,s}^{k,m,q} = \chi_{\ell+2s}^{k-s,q} D_{\ell,s}^{k,m}
\end{equation}
thus gives the formula
\begin{equation*}
    c_{\ell}^{k,m} = (Q_{\ell,0}^{k,m,k})^{-1}\Bigl( \inner{(F\eta)Y_{k+1}^0,Y_{\ell+k+1}^m}_{L^2(\partial B)} - \sum_{q=0}^{k-1}\sum_{s = 0}^{k-q}Q_{\ell,s}^{k,m,q}c_{\ell+2s}^{q,m} \Bigr).
\end{equation*}
Of course, we need to verify that $Q_{\ell,0}^{k,m,k} \neq 0$ for all $\ell,k\in\N_0$ and $m\in\Z_\ell$. Going backwards, from \eqref{eq:Q}, \eqref{eq:D}, \eqref{eq:tau}, and \eqref{eq:chi}, we have
\begin{align}
    Q_{\ell,s}^{k,m,q} &= (-1)^{m+1}\chi_{\ell+2s}^{k-s,q} \tau_{\ell,\ell+2s}^k G_{k+1,\ell+k+1,\ell+2s}^{0,-m,m} \notag \\
    &= (-1)^{m+1}\frac{\sqrt{2\ell+4q+4s+3}(k-s+1)(k-q-s+1)_q}{(k+1)(\ell+k+1)(\ell+k+s+\tfrac{5}{2})_q}G_{k+1,\ell+k+1,\ell+2s}^{0,-m,m}. \label{eq:Q2}
\end{align}
Since $s \leq k-q \leq k$, the only possibility for \eqref{eq:Q2} to vanish is related to the Gaunt coefficient. Since the conditions in \eqref{eq:gauntindices} are satisfied, we have already argued after \eqref{eq:gauntindices} that the only possibility for getting a zero is in the last of the $3j$ symbols in \eqref{eq:gauntconstants}. For $G_{k+1,\ell+k+1,\ell}^{0,-m,m}$ (with $s=0$) the corresponding $3j$ symbol is
\begin{equation} \label{eq:3jfinal}
    \begin{pmatrix}
        k+1 & \ell+k+1 & \ell \\
        0 & -m & m
    \end{pmatrix} = 
    \begin{pmatrix}
        \ell & k+1 & \ell+k+1 \\
        m & 0 & -m
    \end{pmatrix}.
\end{equation}
Here we used \cite[§34.3(ii)~Eq.~34.3.8]{NIST}, that $3j$ symbols are invariant to even permutations of the columns. For the latter $3j$ symbol in \eqref{eq:3jfinal}, then \cite[§34.3(i)~Eq.~34.3.6]{NIST} and $\abs{m}\leq\ell$ imply that it is indeed non-zero, hence concluding the proof of Theorem~\ref{thm:main}. \hfill \qed

\subsection*{Acknowledgements}

We thank Nuutti Hyv\"onen (Aalto University) for useful discussions on the method and its differences to the corresponding 2D method. HG is supported by grant 10.46540/3120-00003B from Independent Research Fund Denmark \textbar\ Natural Sciences. MH is supported by the Academy of Finland (decisions 353081 and 359181).

\appendix

\section{Extension of the linearised problem} \label{sec:extension}

In this appendix we consider a bounded smooth domain $\Omega$ in $\R^d$ for integer $d\geq 2$. Replacing $B$ with $\Omega$ in the conductivity equation \eqref{eq:strong_form}, the corresponding ND map $\Lambda(\gamma)$ is a compact self-adjoint operator in $\mathscr{L}(L^2_\diamond(\partial\Omega))$ and $\gamma\mapsto \Lambda(\gamma)$ is Fr\'echet differentiable with respect to complex-valued perturbations $\eta \in L^\infty(\Omega)$. 

For $F = D \Lambda(1; \eta)$, the Fr\'echet derivative of $\Lambda$ at $\gamma \equiv 1$ with respect to perturbation $\eta$, then $F\in \mathscr{L}(L^\infty(\Omega), \mathscr{L}(L^2_\diamond(\partial\Omega)))$ with
\begin{equation} \label{eq:frechet_gen_redef}
    \inner{(F \eta) f, g}_{L^2(\partial\Omega)} = -\int_\Omega \eta \nabla u_f \cdot \overline{\nabla u_g} \,\di x,
\end{equation}
where $u_f$ and $u_g$ are harmonic functions in $\Omega$ with $f$ and $g$ as their Neumann traces, respectively. 
\begin{proposition} \label{prop:Ld}
    $F$ extends via \eqref{eq:frechet_gen_redef} to an operator in $\mathscr{L}(L^d(\Omega), \mathscr{L}(L^2_\diamond(\partial\Omega)))$.
\end{proposition}
\begin{proof}
    The proof is analogous to that of \cite[Proposition~1.1]{Garde2024a}. Let $\eta\in L^d(\Omega)$ and $f,g\in L^2_\diamond(\partial\Omega)$. According to \cite[Chapter~2, Remark~7.2]{Lions1972}, $u_f,u_g\in H^{3/2}(\Omega)/\C$ with continuous dependence on the Neumann data,~i.e.
    \begin{equation} \label{eq:ufregular}
        \norm{u_f}_{H^{3/2}(\Omega)/\C} \leq C\norm{f}_{L^2(\partial \Omega)} \qquad \text{and} \qquad \norm{u_g}_{H^{3/2}(\Omega)/\C} \leq C\norm{g}_{L^2(\partial \Omega)}.
    \end{equation}
    Using the continuous embedding $H^{1/2}(\Omega) \hookrightarrow L^{\frac{2d}{d-1}}(\Omega)$ (e.g.~\cite[Corollary~4.53]{Demengel2012}) and the generalised H\"older inequality, we may estimate as follows:
    \begin{align*}
        \abs{\inner{(F\eta)f,g}_{L^2(\partial\Omega)}} &= \bigl|\int_\Omega \eta\nabla u_f\cdot\overline{\nabla u_g}\,\di x\bigr| \\[1mm]
        &\leq \norm{\eta}_{L^d(\Omega)}\norm{\nabla u_f}_{L^{\frac{2d}{d-1}}(\Omega)}\norm{\nabla u_g}_{L^{\frac{2d}{d-1}}(\Omega)} \\[1mm]
		&\leq C \norm{\eta}_{L^d(\Omega)}\norm{\nabla u_f}_{H^{1/2}(\Omega)}\norm{\nabla u_g}_{H^{1/2}(\Omega)} \\[1mm]
		&\leq C \norm{\eta}_{L^d(\Omega)}\norm{u_f}_{H^{3/2}(\Omega)/\mathbb{C}}\norm{u_g}_{H^{3/2}(\Omega)/\mathbb{C}}.
    \end{align*} 
    Combining this with \eqref{eq:ufregular} concludes the proof.
\end{proof}

\bibliographystyle{plain}

\end{document}